\documentclass[11pt]{amsart}
\usepackage{amssymb,amsthm,amsmath}
\usepackage{graphicx,tikz}

\newtheorem{theorem}{Theorem}[section]
\newtheorem{proposition}[theorem]{Proposition}
\newtheorem{lemma}[theorem]{Lemma}
\newtheorem{corollary}[theorem]{Corollary}
\theoremstyle{definition}
\newtheorem{definition}[theorem]{Definition}
\newtheorem{example}[theorem]{Example}
\newtheorem{remark}[theorem]{Remark}

\def\K{\mathbb{K}}

\def\R{\mathbb{R}}

\def\TR{\mathbb{TR}}

\def\a{\alpha}

\definecolor{dgreen}{rgb}{.2,.6,.2}
\colorlet{darkgreen}{black!30!dgreen}

\begin{document}
\bibliographystyle{plain}

\title{On real tropical bases and real tropical discriminants}
\author{Luis Felipe Tabera}
\date{}

\address{\small \rm  Departamento de Matem\'aticas, Universidad de Cantabria,
39071, Santander, Spain}
\email{taberalf@unican.es}
\thanks{Supported by the Spanish ``Ministerio de Ciencia e Innovacion'' under 
the Project MTM2011-25816-C02-02.}

\begin{abstract}
We explore the concept of real tropical basis of an ideal in the field of real 
Puiseux series. We show explicit tropical bases of zero-dimensional real 
radical ideals, linear ideals and hypersurfaces coming from combinatorial 
patchworking. But we also show that there exist real radical ideals that do not 
admit a tropical basis. As an application, we show how to compute the set of 
singular points of a real tropical hypersurface. i.e. we compute the real 
tropical discriminant.
\end{abstract}
\maketitle

\section{Introduction}\label{sec:introduction}

In this article, we try to develop an analog of the tropical algebra applied to 
the case of real varieties. Our point of view is to define real tropical 
varieties as \emph{non-archimedean amoebas plus signs}. They are the 
restriction to the reals of the complex tropical curves of G. Mikhalkin in 
\cite{Mik05} and a natural extension of combinatorial patchworking. Different 
approaches to real tropical geometry appear, for instance, in the study of 
logarithmic limits of semialgebraic sets in \cite{Alessandrini-logarithmic} or 
the study of initial real radical ideals in \cite{real-radical-tropical}. 
Another approach with nice combinatorial properties is the study of the 
positive part of a tropical variety \cite{trop-pos-grassman}.

We show that, contrary to the complex case, Kapranov's theorem or, more 
generally, the fundamental theorem of tropical geometry does not hold in the 
real case and there may not be tropical bases. cf. \cite{Computing_trop_var}, 
\cite{Kapranov}, \cite{Lifting-Constr}, \cite{Payne-fibers}, 
\cite{Kapranov-EACA}.

However, for sufficiently simple yet interesting varieties, we can compute real 
tropical basis. We include here real radical zero dimensional ideals, linear 
varieties and hypersurfaces constructed using combinatorial patchworking. As an 
application, we are able to describe the singular locus of a real tropical 
hypersurface using an analogue of the techniques introduced in 
\cite{singular-tropical-hypersurfaces}.

The presentation is done over the field of real Puiseux series $\K$. In 
principle, we could work over any real closed field $\mathbb{F}$ provided with 
a nontrivial valuation and most of our results can be translated to this more 
general setting without trouble. However, some results depending on the residue 
field of $\mathbb{F}$ being archimedean or not, see for instance 
Proposition~\ref{prop:fallaennoarquimediano} and 
Example~\ref{ex:puiseuxdepuiseux}.

The paper is structured as follows. In Section~\ref{sec:preliminares} we 
introduce the notation and basic definitions in real tropical geometry. In 
Section~\ref{sec:basis} we show the main results concerning the existence and 
computation of real tropical basis. Finally, in Section~\ref{sec:singular} we 
apply our results to the computation of the singular locus of a real tropical 
hypersurface.

\section{Preliminaries}\label{sec:preliminares}

Let us introduce some notation. Let $\K$ be the valued field of real Puiseux 
series,
\[\mathbb{K}=\bigcup_{n\geq 1} \mathbb{R}((t^{1/n}))\quad v:\K^*\rightarrow 
\mathbb{Q}\subseteq \mathbb{R}.\]
Every element is a power series with real coefficients and rational exponents 
with bounded denominator
\[p=\sum_{i\geq r_0} a_i t^{i/n},\quad a_{r_0}\neq 0\]
The valuation $v(p)$ of a nonzero power series $p$ is the least exponent $r_0$ 
appearing in the development of the series and the principal coefficient is 
$a_{r_0}$. This principal coefficient is the residue of the series $pt^{-v(p)}$ 
in the residue field $\mathbb{R}$. $\K$ is an ordered field with the order 
given 
by the relation $p>0$ is positive if and only if its principal coefficient 
$a_{r_0}$ is positive. The valuation is compatible with the order in the sense 
that if $0<p<q$ then $v(q)\leq v(p)$. Moreover, since $\K[i]$ is the 
algebraically closed fields of Puiseux series, $\K$ is a real closed field.

\begin{definition}\label{def:real_tropicalization}
The \textbf{real tropicalization} is the valuation taking into account the sign 
of the series:
\[\begin{matrix}trop:&\mathbb{K}^* &\longrightarrow &\TR=\{1,-1\}\times 
\mathbb{R}\\
& x & \mapsto& (s(x),v(x))\end{matrix}\]
where $s(x)$ is the sign function $s:\mathbb{K}^*\rightarrow \{1,-1\}$, $s(x)=1$
if $x>0$ and $s(x)=-1$ if $x<0$. We will sometimes denote by $a^+=(1,a)$ and 
$a^-=(-1,a)\in \TR$. If $x=(p,a)\in \TR$, then $p=s(x)\in\{1,-1\}$ is the 
\textbf{sign} of $x$ and $a=|x|\in\mathbb{R}$ is the \textbf{modulus} of $x$. 
If 
$a=((s_1,a_1),\ldots,(s_n,a_n))\in \TR^n$, we denote by 
$s(a)=(s_1,\ldots,s_n)\in \{1,-1\}^n$, $|a|=(a_1,\ldots,a_n)\in \R^n$, the 
\textbf{sign} 
and \textbf{modulus} taken component-wise.
\end{definition}

\begin{remark}
Note that, while $\TR$ is a group with the \textbf{tropical multiplication} 
$(i,a)\odot(j,b) = (i\cdot j, a+b)$, it is not a tropical semiring, since there 
is not a reasonable definition of addition for $a^+\oplus
a^-$.
\end{remark}

We now define our main geometric objects, real tropical varieties, as the image 
of an algebraic set under the $trop$ map.

\begin{definition}
Let $V\subseteq (\K^*)^n$ be a real variety in the torus. The \textbf{real 
tropicalization} of $V$ is the closure (in $\R^n$) of 
the image $trop(V)\subseteq \TR^n$ of the real points of $V$ under the 
tropicalization map applied component-wise. If $V=\mathcal{V}(I)$ and $I$ is 
generated by polynomials in $\R[x_1,\ldots,x_n]$, then we say that we are 
dealing with the \textbf{constant coefficient case}. 
\end{definition}

This definition is related with taking the non-archimedean amoeba and co-amoeba 
of the set of real points of a real variety (cf. the complex tropical curves 
presented in \cite{Mik05} section 6).

\begin{figure}
\begin{tikzpicture}[xscale=0.55,yscale=0.55]
\draw (-8,0) -- (8,0);
\draw (0,-8) -- (0,8);
\draw[dashed] (-8,8) -- (-1,8);
\draw[dashed] (-8,6) -- (-1,6);
\draw[dashed] (-8,4) -- (-1,4);
\draw[dashed] (-8,2) -- (-1,2);
\draw[dashed] (1,8) -- (8,8);
\draw[dashed] (1,6) -- (8,6);
\draw[dashed] (1,4) -- (8,4);
\draw[dashed] (1,2) -- (8,2);
\draw[dashed] (-8,-8) -- (-1,-8);
\draw[dashed] (-8,-6) -- (-1,-6);
\draw[dashed] (-8,-4) -- (-1,-4);
\draw[dashed] (-8,-2) -- (-1,-2);
\draw[dashed] (1,-8) -- (8,-8);
\draw[dashed] (1,-6) -- (8,-6);
\draw[dashed] (1,-4) -- (8,-4);
\draw[dashed] (1,-2) -- (8,-2);
\draw[dashed] (-8,8) -- (-8,1);
\draw[dashed] (-6,8) -- (-6,1);
\draw[dashed] (-4,8) -- (-4,1);
\draw[dashed] (-2,8) -- (-2,1);
\draw[dashed] (8,8) -- (8,1);
\draw[dashed] (6,8) -- (6,1);
\draw[dashed] (4,8) -- (4,1);
\draw[dashed] (2,8) -- (2,1);
\draw[dashed] (-8,-8) -- (-8,-1);
\draw[dashed] (-6,-8) -- (-6,-1);
\draw[dashed] (-4,-8) -- (-4,-1);
\draw[dashed] (-2,-8) -- (-2,-1);
\draw[dashed] (8,-8) -- (8,-1);
\draw[dashed] (6,-8) -- (6,-1);
\draw[dashed] (4,-8) -- (4,-1);
\draw[dashed] (2,-8) -- (2,-1);
\draw[fill] (4,4) circle(4pt) (-4,4) circle(4pt) (-6,6) circle(4pt);
\draw[fill] (-4,-4) circle(4pt) (-6,-6) circle(4pt) (6,-6) circle(4pt);
\draw[fill] (4,-4) circle(4pt);
\draw[thick] (8,4) -- (4,4) -- (4,8);
\draw[thick] (1.5,1.5) -- (4,4);
\draw[thick] (-6,8) -- (-6,6) -- (-4,4) --(-8,4);
\draw[thick] (8,-6) -- (6,-6) -- (4,-4) --(4,-8);
\draw[thick] (-8,-6) -- (-6,-6) -- (-6,-8);
\draw[thick] (-4,-4) -- (-1.5,-1.5);
\draw (5,4.5) node {\tiny $(0^+,0^+)$};
\draw (-3,4.5) node {\tiny $(0^-,0^+)$};
\draw (-5,6.5) node {\tiny $(1^-,1^+)$};
\draw (-5,-3.5) node {\tiny $(0^-,0^-)$};
\draw (-5,-5.5) node {\tiny $(1^-,1^-)$};
\draw (7,-5.5) node {\tiny $(1^+,1^-)$};
\draw (5,-3.5) node {\tiny $(0^+,0^-)$};
\end{tikzpicture}
\caption{$\mathcal{T}_\R(f)$ with $f = 1^+ \oplus 0^+ v \oplus 0^+ w \oplus 0^- 
v^2 \oplus 0^+ vw \oplus 0^-w^2$}\label{fig:conica_no_triangulation}
\end{figure}
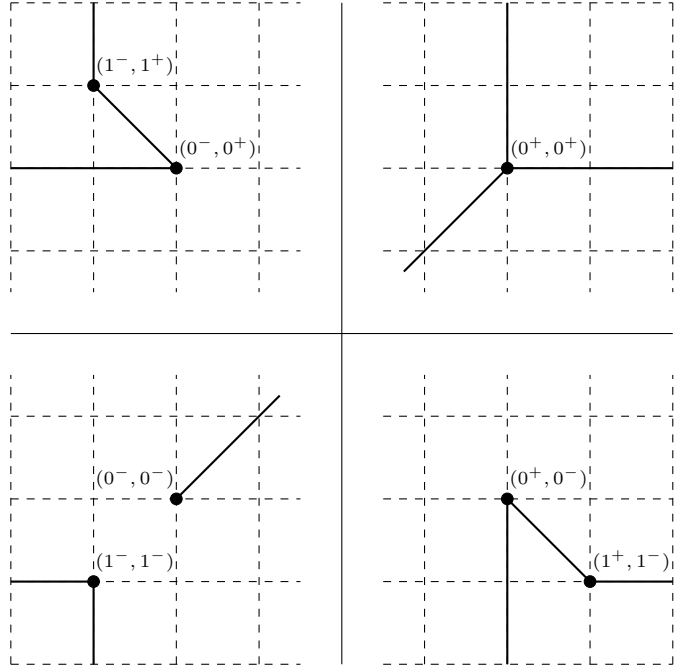

In the algebraically closed case, the tropicalization of a variety $V$
can be computed using Kapranov's theorem, also know as the fundamental theorem
of tropical geometry. The real case seems more involved. The tropicalization in 
the constant coefficient case has been studied in \cite{real-radical-tropical} 
but there is a difference in the tropicalization in the non constant 
coefficient case we are studying here that has its root in P\'olya's Theorem 
\cite{Polyazeros} (See also \cite{positive_polynomial} for a more general 
result).

\begin{theorem}[\cite{Polyazeros}]\label{teo:Polya}
Let $F\in \mathbb{R}[x_1,\ldots, x_n]$ be a homogeneous polynomial that is 
positive in the set $\{x\in \mathbb{R}^n| x_i\geq 0, x_1+\ldots+x_n \neq 0\}$, 
then for $N$ sufficiently large $F\cdot(x_1+\ldots+x_n)^N$ has only positive 
coefficients.
\end{theorem}

However, it is well known that this theorem does not hold if the ground field 
is not archimedean (See Example~\ref{ex:univariate_not_Polya}). This translates 
to the fact that tropicalization of polynomials is not so useful in the 
non-constant coefficient case.

We now introduce a different way of defining tropical varieties in terms of a 
real tropical polynomial. The motivation of this alternative approach is to get 
rid of algebraic polynomials and deal only with tropical polynomials, this is a 
common approach in the tropical case over complex Puiseux series.

\begin{definition}
A \textbf{real polynomial} $f\in \TR[w_1,\ldots, w_n]$ is just a formal sum of 
tropical monomials of the form $f=\oplus_{\ell\in A}a_\ell w^\ell\in 
\TR[w_1,\ldots,
w_n]$, $A\subseteq \mathbb{N}^n$. Every real tropical polynomial defines a 
piecewise affine function $f:\mathbb{TR}^n\rightarrow \mathbb{R}$. 
$f(p)=\min\{|a_\ell\ |\ +\langle |p|,\ell\rangle \ |\ \ell\in A\}$.

We define the \textbf{real tropical 
hypersurface} defined by $f$, $\mathcal{T}_\R(f)$ by $p=(p_1,$ $\ldots,$ 
$p_n)\in \mathcal{T_\R}(f)$ if there are two monomials $i\neq j$ such that
\[s(a_i)\prod_{l=1}^ns(p_l)^{l_i}=1,\quad s(a_j)\prod_{l=1}^ns(p_l)^{l_j}=-1\] 
and
\[|a_i|+\langle i,|p|\rangle=|a_j|+\langle j,|p|\rangle\ \leq\ |a_k|+\langle 
k,|p|\rangle\] for all $k\neq i,j$.

That is, if the minimum of $f(p)$ is attained at two different monomial $i,j$ 
where the evaluation at the point $p$ yields two different signs.
\end{definition}

\begin{example}\label{ex:raices-trop}
Let $f=0^+\oplus 1^+w\oplus 0^+w^2\oplus 1^+w^3\oplus2^-w^4$. We have 
associated the piecewise-linear map $p\mapsto \min\{0, 1+p,0+2p, 1+3p, 2+4p\}$. 
The minimum of this piecewise-linear is attained at least twice for $p=0,-1$. 
The value and sign attained in the candidates of roots are:
\begin{center}
\begin{tabular}{|c|c|c|c|c|c|}
     &$0$  &$1+p$&$0+2p$&$1+3p$&$2+4p$\\\hline
$0^+$&$0^+$&$1^+$&$0^+$&$1^+$&$2^-$\\\hline
$0^-$&$0^+$&$1^-$&$0^+$&$1^-$&$2^-$\\\hline
$-1^+$&$0^+$&$0^+$&$-2^+$&$-2^+$&$-2^-$\\\hline
$-1^+$&$0^+$&$0^-$&$-2^+$&$-2^-$&$-2^-$\\\hline
\end{tabular}
\end{center}
Neither $0^+$ nor $0^-$ are real tropical roots of $f$, because the minimum is 
attained at $0^+$ so we do not have two different signs. On the other hand, 
$-1^+$ and $-1^-$ are both real tropical roots, since the minimum is attained 
at different signs $2^+$ and $2^-$. Hence $\mathcal{T}_\R(f)=\{-1^+,-1^-\}$.
\end{example}

In Figure~\ref{fig:conica_no_triangulation}, we show the four tropical quadrants 
of the conic given by the polynomial $f = 1^+ \oplus 0^+ v \oplus 0^+ w \oplus 
0^- v^2 \oplus 0^+ vw \oplus 0^-w^2$. Note that we are working with the $\min$ 
and that the subdivision induced by $f$ is not a triangulation, so the picture 
is different from the usual ones in patchworking.

\begin{definition}
If $F = \sum_{\ell\in A} a_\ell x^\ell\in \mathbb{K}[x_1,\ldots, x_n]$ is a real
polynomial, write \[trop(F)=f=\bigoplus_{\ell\in A} 
(s(a_\ell),v(a_\ell))w^\ell.\]
By abuse of notation, we will write 
\[\mathcal{T}_\R(F)= \mathcal{T}_\R(f)\subseteq \TR^n\]
\end{definition}
Clearly $trop(\mathcal{V}(F))\subseteq \mathcal{T_\R}(F)$ 
(Example~\ref{ex:univariate_Polya}), but contrary to the algebraically closed 
set, the inclusion may be strict. Next lemma shows further discrepancies 
between the real and usual tropical setting.

\begin{lemma}\label{lem:factorization}
Let $F$, $G\in \K[x_1,\ldots,x_n]$, then 
\[\mathcal{T}_\R(F\cdot G) \subseteq \mathcal{T}_\R(F) \cup \mathcal{T}_\R(G)\]
and the inclusion may be strict.
\end{lemma}
\begin{proof}
Let $p\in \TR^n$ such that $p\notin (\mathcal{T}_\R(F) \cup 
\mathcal{T}_\R(G))$. Without lost of generality, assume that 
$p=(0^+,\ldots,0^+)$. Let $A_1$ (resp $A_2$) be the monomials of $F$ (resp 
$G$) where the minimum of $trop(F)$ (resp. $trop(G))$ is attained at $p$. Then, 
the tropicalization of the monomials of $A_1$ attain the same sign at $p$ and 
so do the monomials in $A_2$. We may also suppose that all these monomials have 
positive coefficients. Then, the monomials of $FG$ where the minimum is 
attained at $p$ are all of the form $ca_1a_2$, with $c\in \K$, $c>0$, $a_1\in 
A_1$, $a_2\in A_2$. It follows that all monomials yield the same sign at $p$ 
and $p\notin \mathcal{T}_\R(FG)$.

To see that the inclusion may be strict, let $F=x^{2} y + x^{2} -  x y -  x + y 
+ 1$, $G=x y^{2} -  x y + y^{2} + x -  y + 1$. Then $FG=x^{3} y^{3} + x^{3} + 
y^{3} + 1$. So $p=(0^+,0^+)$ belongs to both $\mathcal{T}_\R(F)$ and 
$\mathcal{T}_\R(G)$, but not to $\mathcal{T}_\R(FG)$.
\end{proof}

We end this section with an elementary result concerning the number of real 
roots that an 
univariate polynomial admit.

\begin{definition}\label{def:realmultiplicityalaDescartes}
Let $f=\oplus_{\ell\in A} a_\ell w^\ell\in \TR[w]$ be an univariate real 
tropical polynomial. Let $f'=\oplus_{\ell\in A} |a_\ell| w^\ell\in 
\mathbb{T}[w]$ be the usual tropical polynomial obtained from $f$ by forgetting 
signs.

Let $p\in \mathcal{TR}$ with $|p|$  a tropical root of $f'$. Let 
$i_1<\ldots<i_r$ be 
the monomials where $|a_\ell| + \langle |p|, \ell\rangle$ attains its minimum. 
Take the 
sequence of signs
\[S_p=(s(a_{i_1})s(p)^{i_1}, \ldots, s(a_{i_n})s(p)^{i_n})\]
then, the \textbf{complex multiplicity} of $|p|$ in $f'$ is $i_n-i_1$ and the 
\textbf{real multiplicity} of $p$ in $f$ is the number of changes of signs in 
$S_p$.
\end{definition}

\begin{lemma}\label{lem:descartesmanda}
Let $f = \oplus_{\ell\in A}a_\ell w^\ell\in \TR[w]$ be an univariate real 
tropical polynomial of degree $n$. Let $f' = \oplus_{\ell \in A}|a_\ell|w^\ell 
\in \mathbb{T}[w]$ be the (usual) tropical polynomial obtained from $f$ by 
forgetting the signs. Let $p\in \mathbb{T}$ be a (usual) tropical root of $f'$ 
of (usual) multiplicity $m$. Let $m^+$ (resp. $m^-$) be the real tropical 
multiplicity of $p^+$ (resp. $p^-$) as a real tropical root of $f$. Then $m\leq 
m^+ + m^-$ and the difference $m-m^+-m^-$ is an even number.
\end{lemma}
\begin{proof}
Both multiplicities only depend on the support and signs of the coefficients 
where the minimum is attained when evaluating $f$ at $p$. Hence, we may assume, 
without loss of generality, that all coefficients of $f$ and the tropical root 
have modulus 0. consider the polynomial
\[F=\sum_{\ell\in A}s(a_\ell)\cdot t^{\ell^2}x^\ell\in \R[t][x].\]
By combinatorial patchworking \cite{patchworking}, for any sufficiently small 
evaluation $t_0$ of $t$, $0<t_0<<1$, $F(t=t_0)$ will have exactly $m^+$ 
positive roots and $m^-$ negative roots, $m$ equals the number of nonzero roots 
of $F$, so $m^++m^-\leq m$ and its difference is the number of non-real roots 
of $F$ is even.
\end{proof}

\begin{example}
Let $f=0^+\oplus 1^+w\oplus 0^+w^2\oplus 1^+w^3\oplus2^-w^4$ be the real 
tropical polynomial from Example~\ref{ex:raices-trop}, the usual tropical roots 
of $f'$ are $0$ and $-1$ with multiplicity 2 each root. Now, for the root zero, 
we know that $f$ does not admit a real tropical root. So for $|p|=0$ we have 
that 
$m_+=m_-=0$, $m=2$ and $m-m_+-m_-=2$. On the other hand, the sequence of signs 
$S_{1^+}=(1,1,-1)$ so $m^+=1$, $S_{1^-}=(1,-1,-1)$ and $m_-=1$, $m-m_+-m_-=0$ 
is also even.
\end{example}

\section{Real Tropical Basis}\label{sec:basis}

We start with the definition of real tropical basis, cf. 
\cite{Computing_trop_var}

\begin{definition}
Let $I\subseteq \K[x_1,\ldots,x_n]$. A \textbf{real tropical basis} of $I$ is a 
finite set $\{F_1,\ldots, F_r\}$ that generates $I$ such that 
$trop(\mathcal{V}(I)) = \cap_{i=1}^r\mathcal{T_\R}(F_i)$.
\end{definition}

\begin{example}\label{ex:univariate_Polya}
Let $F=x^2-x+1$, $\mathcal{V}(F)=\emptyset$, there are no real solutions. But 
$trop(F)=f=0^+w^2\oplus 0^-w\oplus 0^+$ and $\mathcal{T_\R}(f)=\{0^+\}$. This 
means 
that the generator $\{x^2-x+1\}$ is not a real tropical basis of the ideal 
$(x^2-x+1)$. Let $x^3+1=(x^2-x+1)(x+1)$. $\mathcal{T_\R}(0^+w^3\oplus 
0^+)=\{0^-\}$, hence $\mathcal{T_\R}(f)\cap \mathcal{T_\R}(0^+w^3\oplus 
0^+)=\emptyset$ and $I=(x^2-x+1, x^3+1)$ is a real tropical basis of $I$. Note 
that, in this case, $I$ is not real radical.
\end{example}

The first result is that real-radical zero-dimensional ideals admit a tropical 
basis. 

\begin{definition}
Let $I\subseteq \K[\underline{x}]=\K[x_1,\ldots, x_n]$. The \textbf{real 
radical} 
of $I$, $\sqrt[\R]{I}$ is the ideal
\[\sqrt[\R]{I}=\{F\in \K[\underline{x}]\ |\ \exists G_1,\ldots,G_r\in 
\K[\underline{x}], \exists m>0, F^{2m}+\sum_{i=1}^r G_i^2\in I\}.\]
By the real theorem of zeros \cite{real-alg-geo}, \[\sqrt[\R]{I}=\{F\in 
\K[x_1,\ldots,x_n]\ |\ 
\forall a\in \mathcal{V}_\K(I), F(a)=0\}\]
An ideal $I$ is \textbf{real radical} if $I=\sqrt[\R]{I}$.
\end{definition}

\begin{theorem}
Let $I$ be a real radical zero dimensional ideal, with $\mathcal{V}(I)\subseteq 
(\K^*)^n$ then $I$ admits a tropical basis.
\end{theorem}
\begin{proof}
Let $V=\mathcal{V}(I)=\{p_1,\ldots, p_r\}\subseteq (\K^*)^n$. Let 
$p_i=(p_{i1},\ldots, p_{in})$ and $trop(p_i)=a_i=(a_{i1},\ldots,a_{in})$, 
$1\leq i\leq n$. So $trop(V)=\{a_1,\ldots,a_r\}$. Let $F_j$ be the squarefree 
part of $\prod_{i=1}^r(x_j-p_{ij})$, since $I$ is real radical $F_j\in I, 1\leq 
j\leq n$. The tropical roots of $trop(F_j)$ are precisely $a_{1j},\ldots, 
a_{rj}$  
and the number of tropical roots of each sign and valuation $a_{ij}$ can be 
recovered using Lemma~\ref{lem:descartesmanda} and taking into account that all 
the roots of $F_j$ are real. The set of tropical polynomials $\left\{trop(F_j), 
1\leq j\leq n\right\}$ describe a finite set $S$ of points containing $trop(V)$.

Let $L=b_1x_1+\ldots +b_nx_n$ be a linear function with integer coefficients 
such that $L$ is injective in the set of modulus of $S$, $|S|=\{|a|\ |\ a\in 
S\}$. Let $F_0$ be the squarefree part of the numerator of the polynomial 
$\prod_{i=1}^r (x_1^{b_1}\cdots x_n^{b_n}-p_{i1}^{b_1}\cdots p_{in}^{b_n})$. 
Since this is a polynomial in $I$ whose tropicalization is injective in $S$, it 
allows us to discriminate which modulus of points of $S$ belong to $trop(V)$ or 
not.

Finally, suppose that $c=(c_1,\ldots,c_n)$ is a tropical point in $S\cap 
\mathcal{T}_\R(F_0)$ but not in $trop(\mathcal{V}(I))$. This can only happen 
if there 
is a point in $a\in trop(V)$ with $b\neq a$, $|b|=|a|$ and for every index 
$j$ there is a point in $trop(V)$ with $j$-th coordinate $b_j$.

For every point $a_i\in trop(V)$ with $|a_i|=|c|$, there is a coordinate $h(i)$ 
with $s(a_{i,h(i)})\neq s(c_{h(i)})$, Let $G_c$ be the squarefree part of 

\begin{equation}\label{eq:gc}
\prod_{i, |a_i|=|c|} (x_{h(i)}-p_{i,h(i)}) \times 
\prod_{i,|a_i|\neq|c|}(x_1^{b_1}\cdots x_n^{b_n}-p_{i1}^{b_1}\cdots 
p_{in}^{b_n}) \in I
\end{equation}

This polynomial vanishes by construction on every element of $trop(V)$, but $c$ 
is not a real tropical root of $trop(G_c)$, because, for each factor of $G_c$ 
in 
(\ref{eq:gc}), $c$ either has the wrong modulus or the wrong distribution of 
signs. By Lemma~\ref{lem:factorization}, $c\notin \mathcal{T}_\R(G_c)$.

If $H$ is any generator set of $I$, then $H\cup \{F_j, 1\leq j\leq 
n\}\cup\{F_0\}\cup \{G_c| c\in S-trop(V)\}$ is a real tropical basis of $I$.
\end{proof}

\begin{example}
Let $V=\{(-1,2)$, $(2,3)$, $(-3,-t)$, $(1,-4)$, $(2t,4t)\}\subseteq 
(\mathbb{K}^*)^2$, $trop(V) = \{(0^-,0^+)$, $(0^+,0^+)$, $(0^-,1^-)$, 
$(0^+,0^-)$, $(1^+,1^+)\}$. Let $I=\mathcal{I}(V)$ be the ideal of $V$. We are 
computing a real tropical basis of $I$. First, define
\[F_1(x)=x^{5} + \left(-2 t + 1\right) x^{4} + \left(-2 t - 7\right) x^{3} + 
\left(14 t - 1\right) x^{2} + \left(2 t + 6\right) x - 12 t,\]
\[F_2(y)=y^{5} + \left(-3 t - 1\right) y^{4} + \left(-4 t^{2} + 3 t - 14\right) 
y^{3} + \left(4 t^{2} + 42 t + 24\right) y^{2} + \]\[\left(56 t^{2} - 72 
t\right) y - 96 t^{2}.\]
$trop(F_1)$ and $trop(F_2)$ define the finite set
\[\begin{matrix}
   S=&\{(0^+,0^+),&(0^+,1^+), &(0^+,1^-), &(0^+,0^-),\\
     &\ (0^-,0^+), &(0^-,1^+), &(0^-,1^-), &(0^-,0^-),\\
     &\ (1^+,0^+), &(1^+,1^+), &(1^+,1^-),&(1^+,0^-)\}
\end{matrix}\]
$|S|=\{(0,0), (0,1), (1,0), (1,1)\}$. Note 
that $(1,0)\notin |trop(V)|$. Let $L=$ $x+2y$ by a linear function injective in 
$|S|$. Then
\[F_0=(xy^{2} - 18) \cdot (x y^{2} - 16) \cdot (x y^{2} + 4) \cdot 
(x y^{2} + 3 t^{2}) \cdot (x y^{2} - 32 t^{3})\]
In usual tropical geometry $\{F_1,F_2,F_0\}$ would suffice to provide a 
tropical basis, See \cite{TR_bas_reg_proy}, \cite{Resultantes-trop}. But these 
polynomials are not enough in the real case. The elements of $S$ that are also 
in $\mathcal{T}_\R(F_0)$ are 
$\{(0^+,0^+)$, $(0^+,0^-)$, $(0^-,0^+)$, $(0^-,1^+)$, $(0^-,1^-)$, $(0^-,0^-)$, 
$(1^+,1^-)$, $(1^+, 1^+)\}$. We have to compute polynomials to discard the 
points in $S$ not in $trop(V)$. Following the theorem:\\
$G_{(0^-,0^-)}=(x-2)(x-1)(y-2)(xy^2-32t^3)(xy^2+3t^2)$, and $(0^-,0^-) \notin 
\mathcal{T}_\R(G_{(0^+,0^+)})$. Define also 
$G_{(0^-,1^+)}=(y+t)(xy^2 + 4)(xy^2 - 18)(xy^2 - 16)(xy^2 - 32t^3)$ and
$G_{(1^+,1^-)}=(y-4t)(xy^2 + 4)(xy^2 - 18)(xy^2 - 16)(xy^2 + 3t^2)$. Then, if 
we add to any finite set of generators of $I$ the set $\{F_1$, $F_2$, $F_0$, 
$G_{(0^-,0^-)}$, $G_{(0^-,1^+)}$, $G_{(1^+,1^-)}\}$ we get a real tropical 
basis of $I$.
\end{example}

The hypothesis of being real radical cannot be avoided as the following example 
shows.

\begin{example}\label{ex:univariate_not_Polya}
Let $F=x^2-(2+t)x+1$. The discriminant of $F$ is $t^2 + 4t>0$, so there are no 
real roots of $F$ and $\sqrt[\R]{(F)}=(1)$. However, $I=(F)$ has no real 
tropical 
basis. 
To see that, $\{0^+\}$ is a root of $trop(x^2-(2+t)x+1)=0^+x^2\oplus 0^-x\oplus 
0^+$. Now $v(F(1))=v(-t)=1>0$, $1$ is a real root of $x^2-2x+1$. If $G\in 
\K[x]$, assume without loss of 
generality that the minimum of the valuation of the coefficients of $G$ is 
zero. 
Then $1$ will be a root of the residue polynomial of $F\cdot G$. This means 
that 
$0^+$ will also be a tropical root of $trop(FG)$ and $\bigcap_{H\in 
I}\mathcal{T}_\R(H)=\{0^+\}$.
\end{example}

The reader may think that the failure here is due to the fact that the 
\emph{residue} polynomial $x^2-2x+1$ obtained by substituting $t=0$ has a 
double root. This is true for the field of real Puiseux series but it may fail 
over more complicated fields.

\begin{definition}
Let $F=\sum_{\ell\in A}p_\ell x^\ell\in \K[x_1,\ldots,x_n]$ be a polynomial, 
let $w\in \R^n$. Let $i_1,\ldots,i_r$ be the monomials where $trop(F)$ attains 
its 
minimum at $w$. We define the \textbf{residue polynomial} of $F$ with respect 
to 
$w$ to the polynomial
\[F_w=\sum_{j=1}^r Pc(p_{i_j})x^{i_j}\in \R[x_1,\ldots,x_n],\]
where $Pc(p_{i_j})\in \R$ is the principal coefficient of the power series 
$p_{i_j}$.
\end{definition}

\begin{proposition}\label{prop:fallaennoarquimediano}
Let $F\in \K[x]$ be an univariate polynomial such that, for every real root 
$p\in \K^*$ we have that $F_{trop(p)}$ has only simple roots, then $(F)$ admits 
a tropical basis.
\end{proposition}
\begin{proof}
Let $\{p_1,\ldots,p_r\}=\mathcal{T}_\R(F)$ be the real tropical roots of 
$trop(F)$. If $p_i\notin trop(\{F=0\})$ we have to find a certificate of this 
fact. Without loss of generality, we may assume that $p_i=0^+$. $F_0\in \R[x]$ 
is an univariate polynomial that has no positive real root. From 
Theorem~\ref{teo:Polya}, there exists a natural $N$ such that $F_0\cdot(1+x)^N$ 
has all its coefficients positive. It follows that $0^+ \notin 
\mathcal{T}_\R(F\cdot (1+x)^N)$. Iterating over every element in 
$\mathcal{T}_\R(F)-trop(\{F=0\})$ we can construct a real tropical basis of 
$(F)$.
\end{proof}

We now show that Proposition~\ref{prop:fallaennoarquimediano} cannot be 
extended to any real closed valued field.

\begin{example}\label{ex:puiseuxdepuiseux}
Let $\mathbb{F}=\K\{\{s\}\}$ be the field of Puiseux series is $s$ whose 
coefficients are Puiseux series in $t$ over the reals. Take again 
$F=x^2-(2+t)x+1\mathbb{F}[x]$. Now, $trop(F)$ still has a tropical root $0^+$ 
but 
$F$ has no root in $\mathbb{F}$. By the same argument as in 
Example~\ref{ex:univariate_not_Polya} $(F)$ does not admit a tropical basis. 
Note that in this new context $F=F_0$ has no multiple root.
\end{example}

Let us consider now other ideals that also admit tropical basis. We can revisit 
combinatorial patchworking (See \cite{patchworking} for the details) from the 
point of view of tropical basis.

\begin{theorem}\label{teo:T-construction}
Let $F\in \mathbb{K}[x_1,\ldots, x_n]$ be a real polynomial and let $f = 
trop(F)$ be the corresponding tropical polynomial. $f$ defines a mixed 
subdivision in $Supp(F)$ by duality. Assume that this subdivision is a 
triangulation that contains all the monomials in $Supp(F)$ as vertices. Then 
$\{F\}$ is a real tropical basis of $(F)$. That is 
$\mathcal{T}_\R(F)=trop(\mathcal{V}(F))$.
\end{theorem}
\begin{proof}
We are going to prove that if $p\in \mathcal{T}_\R(F)\cap \mathbb{Q}^n$ then 
$p\in trop(\mathcal{V}(F))$. Without loss of generality, we may assume that 
$p=(0^+,\ldots, 0^+)$.

Let $S=Supp(F)$, the set of monomials of $S$ where $f$ attains its minimum at 
$p$ is a simplex $C$ in $S$ that is not a point. Since $p\in 
\mathcal{T}_\R(f)$, then there is a positive 
monomial $a\in C$ and a negative monomial $b\in C$. Let 
$L=r_0+r_1x_1+\ldots+r_nx_n$ 
be an affine function with integer coefficients such that for all monomials 
$c\in S-\{a,b\}$, $L(a)<L(c)<L(b)$. Let $c_ax^a$, $-c_bx^b$ be the 
corresponding 
monomials of $F$. Consider the points of the form $(s^{r_1},\ldots, s^{r_n})$. 
If $0<s_-<<1$ is a valuation zero positive element of $\K$ small enough, then 
$F(s_-^{r_1},\ldots, s_-^{r_n})\sim c_as^{L(a)}>0$. If $1<<s_+$ is a valuation 
zero element big enough, $F(s_+^{r_1},\ldots, s_+^{r_n})\sim -c_bs^{L(b)}<0$. 
Hence, by the intermediate value theorem, the polynomial $F(s^{r_1},\ldots, 
s^{r_n})$ has a root $s_0$ in the 
interval $[s_-,s_+]$, and hence $s_0$ is of valuation zero. By construction, 
$(s_0^{r_1},\ldots, s_0^{r_n})$ is a zero of $F$ with tropicalization 
$(0^+,\ldots, 0^+)$.
\end{proof}

\begin{corollary}\label{cor:Kapranov_hyperplane}
Real vector hyperplanes have tropical basis. More precisely, let 
$H=V(a_1x_{i_1}+\ldots+a_rx_{i_r}+a_0)\in \K[x_1,\ldots, x_n]$ be a hyperplane 
in $(\mathbb{K}^*)^n$.
Then $trop(H)=\mathcal{T}(trop(a_1)w_{i_1}\oplus\ldots\oplus 
trop(a_n)w_{i_r}\oplus trop(a_0))$.
\end{corollary}
\begin{proof}
We are in the hypothesis of Theorem~\ref{teo:T-construction} since $Supp(H)$ is 
a 
simplex.
\end{proof}

The last goal of this section is to prove that linear ideals have a tropical 
basis (cf. \cite{Bergman-complex, Computing_trop_var} for the algebraically
closed case). We start with a Lemma that states that there are 
infinite tropical basis.

\begin{lemma}\label{lem:trop-basis-linear}
Let $V \subseteq (\mathbb{K}^*)^n$ be a linear space, then \[trop(V) = 
\bigcap_{\substack{\subseteq H\\H\in \mathcal{H}}} trop(H),\] where 
$\mathcal{H}$ is the set of hyperplanes of $(\mathbb{K}^*)^n$.
\end{lemma}
\begin{proof}
Clearly $trop(V)\subseteq \bigcap_{V\subseteq H} trop(H)$. We prove the 
equality by induction in the dimension $n$ of the ambient space. For $n=2$, $V$ 
can only be a line, so the result holds by 
Corollary~\ref{cor:Kapranov_hyperplane}. Assume it is true for ambient 
dimension up to $n-1$. We prove that the result is true in ambient dimension 
$n$ by induction in the codimension of $V$. If $codim(V)=1$, then $V$ is a 
hyperplane and this is again Corollary~\ref{cor:Kapranov_hyperplane}. Assume 
now that the result is true for all codimension $r$ spaces and that $V$ has 
codimension $r+1$. If $V\subseteq \{x_i=0\}$ for some $i$ then 
$trop(V)=trop(\{x_i=0\})=\emptyset$. Let $p=(p_1,\ldots,p_n)\in 
\bigcap_{V\subseteq H} trop(H)$. Without loss of generality 
$p=(0^+,\ldots,0^+)$. If $e_i\in V$ is the $i$-th vector of the canonical 
basis, the $i$-th coordinate of the points of $V$ may take any value and for 
all $V\subseteq H$, the $i$-th coefficient of the defining linear equation of 
$H$ is zero. It follows that we can project along the $i$-th coordinate and $p 
\in trop(V)$ iff $\pi_i(p)\in trop(\pi_i(V))$ and clearly $\pi_i(p)\in 
\bigcap_{V \subseteq H} trop(\pi_i(H)) =\bigcap_{\pi_i(V)\subseteq H} trop(H)$ 
where $H$ is any hyperplane in dimension $n-1$ containing $\pi_i(V)$ and we are 
done by induction.

Hence, we may assume that $e_i\notin V$ for $1\leq i\leq n$. Consider the linear
spaces $V+<e_i>$, $1\leq i\leq n$. Then $p\in \bigcap_{V+<e_i>\subseteq H}
trop(H)=trop(V+<e_i>)$ by induction hypothesis in the codimension of $V$.
Hence, for each $i$, $1\leq
i\leq n$ there is a point $a_i=(a_{i1},\ldots, a_{in}) \in V$ such that
for all $i\neq j$, $a_{ij}>0$ and $v(a_{ij})=0$. We can write the matrix
\[A=\begin{pmatrix}a_{11} & a_{12}&\ldots & a_{1n}\\&&\ldots\\a_{n1} &
a_{n2}&\ldots & a_{nn}\end{pmatrix}\ trop(A)=\begin{pmatrix}q_1 & 0^+&\ldots &
0^+\\0^+&q_2&\ldots &0^+\\ 0^+&0^+&\ldots &q_n\end{pmatrix}\]

Since $codim(V)>1$, $rank(A)\leq n-2$ and $A$ is singular, so it cannot happen 
that for all $i$ $|q_i|< 0$. Assume that $|q_1| \geq 0$. Since the minor 
$A_{11}$ is also singular, there must be another element $|q_i| \geq 0$. Assume 
$|q_2|\geq 0$. If for some $i$ is $|q_i| > 0$ let $\lambda$ be an element of 
valuation zero such that its residue coefficient $b$ is $0<b<<1$ then $\lambda 
a_j +a_i$, $j\in {1,2}, j\neq i$, $trop(\lambda a_j +a_i)=p$. Thus, we may 
assume that $|q_i|\leq 0$ for all $i$. But again, since $A$ is singular, 
$|q_i|=0$ for all $i$. If for some $i$ is $a_{ii}>0$ we are done and $a_i$ is a 
lift of $p$. Last, we have the case that $q_i=0^-$ for all $i$. In this case, 
we perform Gauss reduction on the matrix $A$. The coefficients we have to 
multiply $a_1$ in order to make zeros below the first column are $0<\lambda_i$, 
$v(\lambda_i)=0$. For all the elements $a_{ij}$, $i\neq 1$ $a_{ij}+\lambda_i 
a_{1j}$ is a positive element of valuation $0$. On the other hand 
$v(a_{ii}+\lambda a_{1i})\geq 0$.

Now, if for some $i\neq 1$, $v(a_{ii}+\lambda_i a_{1i})>0$, let $k\notin 
\{1,i\}$, let $0<\eta$ with, $v(\eta)=0$, $0<Pc(\eta)<<1$ then $b=\eta a_k + 
(a_i+\lambda_ia_1)$ is a vector $b\in V$ such that $b_i>0$ and $v(b_i)=0$ and 
we 
are done. 

Next case is if for some $i\neq 1$, $a_{ii}+\lambda_ia_{1i}$ is a positive 
element of valuation $0$, then for $0<\mu$ of valuation 0, $1<<Pc(\mu)$, 
$-a_1+\mu(a_i+\lambda_i a_1)$ is the desired point.

Finally, if for all $i$, $a_{ii}+\lambda_ia_{1i}$ is a
negative element of valuation $0$, then we perform Gauss reduction on the
second column and repeat the process. Since $rank(A)\leq n-2$, after a finite
number of steps, we must arrive to a matrix (after reordering the indices) of
the form
\[\begin{pmatrix}a_{11} & a_{12} & \ldots & a_{1k} & a_{1,k+1} & \ldots &
a_{1n}\\
                0      & b_{22} & \ldots & b_{2k} & b_{2,k+1} & \ldots &
b_{2n}\\
\multicolumn{7}{c}{\ldots}\\
0 & 0 & \ldots & b_{k-1,k-1} & b_{k-1,k} & \ldots & b_{k-1,n}\\
0 & 0 & \ldots & 0 & b_{k,k} & \ldots & b_{k,n}\\
\multicolumn{7}{c}{\ldots}\\
0 & 0 & \ldots & 0 & b_{n,k} & \ldots & b_{nn}\\
\end{pmatrix}\]

with real tropicalization

\[\begin{pmatrix}0^- & 0^+ & \ldots & 0^+ & 0^+ & \ldots & 0^+\\
                 \infty   & 0^- & \ldots & 0^+ & 0^+ & \ldots & 0^+\\
\multicolumn{7}{c}{\ldots}\\
\infty & \infty & \ldots & 0^- & 0^+ & \ldots & 0^+\\
\infty & \infty & \ldots & \infty & c_{k,k}^{+} & \ldots & c_{k,n}^+\\
\multicolumn{7}{c}{\ldots}\\
\infty & \infty & \ldots & \infty & c_{n,k}^{+} & \ldots & c_{nn}^+\\
\end{pmatrix}\]

Such that $k=dim(V)\leq n-2$; $b_{kk}\neq 0$; $b_{ij}=0$ if $i>j<k$; 
$trop(b_{ij})=0^+$,
if $i<j$ or $i>j\geq k$; $trop(b_{ii})=0^-$ for
$i<k$; vectors $b_{k+1},\ldots b_{n}$ are all multiple of $b_k$. Hence, it must
happen that $trop(b_{kk})=0^+$. Then, for $\eta_2,\ldots, \eta_k$
of valuation zero and positive such that $1<<Pc(\eta_2) << Pc(\eta_3)<<\ldots <<
Pc(\eta_k)$ the vector $-b_1-\eta_2b_2-\ldots -\eta_{k-1}b_{k-1}+\eta_kb_k$ is a
lift of $p$ in $V$.
\end{proof}

\begin{theorem}\label{teo:minimal_support_is_tropical basis}
Let $I$ be the ideal of an affine space $V$ in $(\mathbb{K}^*)$. The affine 
polynomials in $I$ with minimal support form a real tropical basis of $I$.
\end{theorem}
\begin{proof}
For each minimal support there is only one linear polynomial in $I$ up to 
multiplication by a constant and they generate $I$. Call $\{F_1,\ldots, F_r\}$ 
this generator set of $I$. Clearly, $trop(V)\subseteq \cap_{i=1}^r 
\mathcal{T}_\R(F_i)$. To see the equality, let $p\notin trop(V)$. Homogenizing 
if necessary and performing a monomial change of coordinates, we may assume, 
without loss of generality that $I$ is homogeneous and that $p=(0^+,\ldots, 
0^+)$. By Lemma~\ref{lem:trop-basis-linear}, there is a linear function $F\in 
V$ such that $p\notin \mathcal{T}_\R(F)$. It suffices to check that we can take 
$F$ with minimal support. If $F$ has no minimal support, then there exists 
another linear $G\in I$ whose support is contained in the support of $F$. 
Multiply $F$ and $G$ by appropriate constants so that the minimum valuation of 
the coefficients of $F$ and of $G$ is zero. By substituting $G$ with a linear 
combination of $F$ and $G$, we may also assume that the residue polynomials at 
$p$ of $F$ and $G$ are linearly independent forms in $\R[x_1, \ldots, x_n]$. 
Let $S_F$ (resp. $S_G$) be the variables where $F$ (resp. $G$) attains the 
minimum at $p$. If there is a monomial $g_mx_m$ in $S_G$ that is not in $S_F$, 
we can use $G$ to make zero the monomial $x_m$ in $F$ without modifying the 
residue polynomial of $F$ at $p$ and proceed recursively.

Assume then that $S_G\subseteq S_F$, after reordering, we may assume that 
$S_F=\{x_{i_1},\ldots, x_{i_r}\}$. Let $a_i$ (resp. $b_i$) be the principal 
coefficient $F$ (resp. $G$) at the monomial $x_i$. Let $i$ be an index in $S_F$ 
such that $|b_i/a_i|$ is maximum of all the indices. If $b_i>0$ substitute $G$ 
by $-G$ so that $b_i<0$. Then, the polynomial $H=F\cdot(-bi/ai)+G$, 
$Supp(H)\subsetneq Supp(F)$ and the residue polynomial of $H$ at $p$ has only 
positive coefficients. We proceed recursively until we arrive to a linear 
polynomial $H'$ with minimal support in $I$ such that $p\notin 
\mathcal{T}_\R(H')$.
\end{proof}

We finish this section showing a real radical ideal that has no tropical basis. 
This \textbf{dissonance} has been studied in \cite{real-radical-tropical}, but 
the example showed there is not real radical.

\begin{example}\label{ex:real_radical_mala_trop}
Consider the cubic defined by:

\[F=x^3+y^3-x^2y-xy^2+2x^2+2y^2+4xy-8x-8y+8\]

The only point with positive coefficients is the singular point $(1,1)$ of 
valuation $(0,0)$.

The corresponding tropical polynomial is
\[f = 0^+v^3\oplus 0^+w^3\oplus 0^-v^2w\oplus 0^-vw^2\oplus 0^+2v^2\oplus 
0^+w^2\oplus 0^+vw\oplus 0^-v\oplus 0^-w\oplus 0^+\]

In the positive tropical quadrant the tropical polynomial defines a line-like 
tropical curve.

\begin{center}
\begin{tikzpicture}
\draw (1,0)--(0,0)--(0,1);
\draw (0,0) -- (-0.7,-0.7);
\draw[fill] (0,0) circle(2pt);
\draw (4,0) -- (3.3,-0.7);
\draw[fill] (4,0) circle(2pt);
\end{tikzpicture}\\
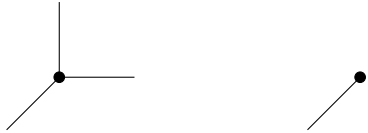
Positive part of a tropical singular cubic (left)\\ and the set of points in 
every tropicalization of a polynomial in $(F)$ (right).
\end{center}

We claim that the curve $C=\{F=0\}$ has no real tropical basis. Since the only 
positive point of the curve is the singular point $(1,1)$, then the only 
tropical positive point is the tropical point $(0^+,0^+)$. The ideal of $C$ is 
$(F)$ and this is a real radical ideal. We are showing that, for every $a<0$, 
$(a^+,a^+)\in \mathcal{T}_\R(FG)$ for every $G\neq 0$, $G\in \K[x,y]$.

Let $G$ be any real polynomial and $a<0$. Let $G_a$ be the residue polynomial 
of $G$ at the point $(a,a)$. Then, 
$(FG)_{(a,a)}=(x^3+y^3-2xy^2-2x^2y)G_a=(x+y)(x-y)^2G_a$. Since $(1,1)$ is a 
root 
of this polynomial, there must be two monomials in $(FG)_{(a,a)}$ with 
different sign 
where $(a,a)$ attains its minimum. That is $(a^+,a^+)\in 
\mathcal{T}_\R(FG)$.

On the other hand, let us take a look at the points of the form $(a^+,0^+)$ and 
$(0^+,a^+)$ with $a>0$. The monomials where the minimum of $(0,a)$ is attained 
are $F_{(0,a)}=x^3 + 2x^2 - 8x + 8$. This polynomial have no positive root, 
hence, by P\'olya's theorem $F_{(0,a)}(1+x)^n$ has only positive monomials.
$(F(1+x)^{11})_{(0,a)}= x^{14} + 13 x^{13} + 69 x^{12} + 195 x^{11} + 308 
x^{10} 
+ 242 x^{9} + 66 x^{8} + 198 x^{7} + 825 x^{6} + 1441 x^{5} + 1441 x^{4} + 903 
x^{3} + 354 x^{2} + 80 x + 8$.

Hence, in the tropical positive quadrant
\[\bigcap_{G}\mathcal{T}_\R(FG) = \{(a^+,a^+)| a\leq 0\}\neq 
trop(C)=\{(0^+,0^+)\}\]
\end{example}

\section{The Real Tropical Discriminant}\label{sec:singular}

In this section we extend the study of tropical singularities and discriminants 
\cite{tropical-discriminant}, \cite{singular-tropical-hypersurfaces}, 
\cite{sing-fixed-point} to the real case.

Consider the algebraic closed field $\K[i]$. Let $A\subseteq \mathbb{Z}^n$. 
$|A|=d$ such that $<A>=\mathbb{Z}^n$. Consider the Laurent polynomial ring 
$\mathbb{K}[i][x_1^{\pm 1},\ldots, x_n^{\pm 1}]$. In this ring we have the 
$\mathbb{K}[i]$-linear space of polynomials with support $A$, $F=\sum_{i\in 
A}a_i x^i$. We may identify any such polynomial with the point $(a_i|i\in A)\in 
\mathbb{K}[i]^d$. Consider the incidence variety
\[H = \{(F,u)\in (\mathbb{K}[i]^*)^d\times (\mathbb{K}[i]^*)^n\ |\ F \textrm{ 
is singular at } u\}\]
of singular hypersurfaces of support $A$ and singular points. This is a 
$\mathbb{Q}$-defined variety. Consider the projections 
$\pi_1:(\mathbb{K}[i]^*)^d \times (\mathbb{K}[i]^*)^n\rightarrow 
(\mathbb{K}[i]^*)^d$ and $\pi_2: (\mathbb{K}[i]^*)^d\times (\mathbb{K}[i]^*)^n 
\rightarrow (\mathbb{K}[i]^*)^n$. $\pi_1(H)$ is the $A$-discriminant variety of 
hypersurfaces with support $A$ and a singular point, while $\pi_2(H) = 
(\mathbb{K}^*)^n$, and the fiber over any point is a linear space isomorphic to 
$\pi_2^{-1}(1,\ldots, 1)=\{F\ |\ F$ is singular at $(1,\ldots,1)\}$. Our aim is 
to describe the real tropicalization of $\pi_1(H)$.

\begin{definition}
Let $f\in \TR[w_1,\ldots, w_n]$ be a tropical polynomial with signs. 
Let $p=(p_1,\ldots, p_n)\in \mathcal{T}_\R(f)$. we say that $p$ is a 
\textbf{singular point} of $f$ if there exists a pair $F\in 
\mathbb{K}[x_1,\ldots, x_n]$ and $P\in \K^n$ such that $P$ is a singular point 
of $\{F=0\}$, $trop(F)=f$ and $trop(P)=p$.
\end{definition}

The definition of tropical Euler derivative 
\cite{singular-tropical-hypersurfaces} can be easily extended to tropical 
polynomials with signs.

\begin{definition}
Let $f = \bigoplus_{\ell\in A} a_\ell w^\ell\in \TR[w_1,\ldots, w_n]$ 
and $<A>=\mathbb{Z}^n$ be a tropical polynomial with signs. Let 
$L=b_0+b_1w_1+\ldots + b_nw_n$ be an affine function with integer coefficients. 
The \textbf{Euler derivative} of $f$ with respect to $L$ is
\[\bigoplus_{\substack{\ell \in A\\ L(\ell)\neq 0}} s(L(\ell)) a_\ell w^\ell.\]
We eliminate all the monomials in the support where $L$ vanish and swap signs 
of the monomials where $L$ is negative. If $F\in \K[x_1,\ldots, x_n]$ is any 
polynomial such that $trop(F)=f$  then
\[\frac{\partial f}{\partial L} = T\left(\frac{\partial F}{\partial 
L}\right)=T\left (b_0F + b_1x_1\frac{\partial 
F}{\partial x_1}+\ldots +b_nx_n\frac{\partial F}{\partial x_n}\right )\]
\end{definition}

\begin{example}
Let $f=0^+ \oplus 0^+w_1 \oplus 0^+w_2 \oplus 0^+ w_1^2\oplus 0^+w_1w_2 \oplus 
0^+w_2^2$. Then $\frac{\partial f}{\partial w_1-w_2} = 0^+w_1 \oplus 0^-w_2 
\oplus 0^+ w_1^2 \oplus 0^-w_2^2$
\end{example}

We have the following consequence of 
Theorem~\ref{teo:minimal_support_is_tropical basis}

\begin{theorem}
Let $f$ be a tropical polynomial with signs. The set of tropical singularities 
of $\mathcal{T}_\R(f)$ is
\[\bigcap_{L} \mathcal{T}_\R\left(\frac{\partial f }{\partial 
L}\right)\]
\end{theorem}
\begin{proof}
If $p$ is a singularity of $\mathcal{T}_\R(f)$, then let $P\in (\K^*)^n$ and 
$F\in \K[x_1,\ldots, x_n]$ be a polynomial, with $trop(P)=p$, $trop(F)=f$, and 
$P$ 
is in the singular locus of $\{F=0\}$. Then, for any $L$, $\frac{\partial 
F}{\partial L}(P)=0$. So $p\in \mathcal{T}_\R(\frac{\partial f}{\partial L})$.

Conversely, let $p\in \bigcap_{L} \mathcal{T}_\R\left(\frac{\partial f 
}{\partial 
L}\right)$. Without loss of generality, assume that $p=(0^+,\ldots, 0^+)$. 
The set tropical polynomials with signs having a singularity at $p$ and support 
$A$ is a linear space that is the tropicalization of the linear system $H_1$ of 
polynomials with support $A$ in $\K[x_1,\ldots, x_n]$ having a singularity at 
$P_1=(1,\ldots, 1)$. Let $F=\sum_{\ell\in A}a_\ell x^\ell$ be the generic 
polynomial with indeterminate coefficients and support $A$. The linear system 
$H_1$ is generated by $F(P_1)$, $(x_1\frac{\partial F}{\partial x_1})(P_1)$, 
$\ldots$, $(x_n\frac{\partial F}{\partial x_n})(P_1)$. By 
Lemma~\ref{lem:trop-basis-linear}, 
\[trop(H_1)=\bigcap_{L}\mathcal{T}_f(\frac{\partial F}{\partial 
L})=\bigcap_{L}\mathcal{T}_f(\frac{\partial f}{\partial L})\]
Hence, there exists $F\in H_1$ with a singularity at $(1,\ldots, 1)$ and $p$ is 
a singularity of $f$.
\end{proof}

\begin{definition}[cf. 
\cite{Bergman-complex},\cite{tropical-discriminant},\cite{sing-fixed-point},
\cite{singular-tropical-hypersurfaces}]
Let $p$ be a point in $\mathcal{T}_\R(f)$. We define the \textbf{flag} of $f$ 
with respect to $p$ as the flag of subsets $\mathcal{F}(p)$ of $A$ defined 
inductively by:
$\mathcal{F}_{-1}=\emptyset \subsetneq \mathcal{F}_0(p) \subsetneq 
\mathcal{F}_1(p) \subsetneq\ldots\subsetneq \mathcal{F}_r(p)$, $dim 
<\!\mathcal{F}_r (p)\!> = n$, and for any $\ell$: $\mathcal{F}_{\ell +1} (p) -
\mathcal{F}_\ell (p)$ is
the subset of $A- <\!\mathcal{F}_\ell(p)\!>$ where the tropical polynomial 
$\oplus_{j \in A-<\!\mathcal{F}_\ell\!>} a_j w^j$ attains its minimum at 
$p$. The weight class of the flag $\mathcal{F}(p)$ are all the points $p'\in 
\mathcal{T}_\R(f)$ for which $\mathcal{F}(p)=\mathcal{F}(p')$. If $p$ and $p'$ 
are in the same weight class and $s(p)=s(p')$ then $p$ is singular if and only 
if $p'$ is singular.
\end{definition}

Let $f\in \TR[w_1,\ldots, w_d]$ and $p=(p_1,\ldots, p_d)\in \mathbb{R}$ be a 
singular point of the usual tropical variety defined by $f$ forgetting signs 
(refer to \cite{singular-tropical-hypersurfaces}.  We want to know if 
$p^+=(p_1^+,\ldots, p_n^+)$ is a singular point of $\mathcal{T}_\R(f)$. Let 
$\mathcal{F}_{-1}= \emptyset \subsetneq \mathcal{F}_0(p)\subsetneq 
\mathcal{F}_1(p)\subsetneq \ldots\subsetneq 
\mathcal{F}_r(p)$ be the \textbf{weight class} of $p$. For any $\mathcal{F}_i$, 
write $\mathcal{F}_i=\mathcal{F}_i^+\cup \mathcal{F}_i^-$, where 
$\mathcal{F}_i^+$ (resp. $\mathcal{F}_i^-$) are the monomials whose 
corresponding tropical coefficient is positive (resp. negative).

\begin{definition}
Let $A,B$ disjoint sets. Let $L$ be a hyperplane, we say that $L$ 
\textbf{separates} $A$ and $B$ if $L$ does not contain $A\cup B$ and the sets 
are contained in different closed halfspaces defined by $L$. That is $(A\cup 
B)\not \subseteq L$, $A\subseteq L_{\geq 0}$ and $B\subseteq L_{\leq 0}$ (or 
$A\subseteq L_{\leq 0}$ and $B\subseteq L_{\geq 0}$).
\end{definition}

Note, that, for example, if $A\subseteq L$ and $B\subseteq L_{\geq0}$ then $L$ 
separates $A$ and $B$. Also, if $A=\emptyset$, any $L$ not containing $B$ 
separates $A$ and $B$.

\begin{theorem}
With the previous notation, let $p\in \TR^n$ be a point with all its 
coordinates 
positive. Then, $p$ is a singular point of $\mathcal{T}_\R(f)$ if and only if 
for all $i=0,\ldots, r-1$ and for all hyperplane $L$ such that 
$\mathcal{F}_{i-1}(p)\subseteq L$, $\mathcal{F}_{i}(p)\not\subseteq L$, $L$ 
does 
not separate $\mathcal{F}_i^+$ and $\mathcal{F}_i^-$.
\end{theorem}
\begin{proof}
Let $L$ be a hyperplane such that $\mathcal{F}_{i-1}(p)\subseteq L$ and $L$ 
separates $\mathcal{F_i}^+$ and $\mathcal{F_i}^-$. Then, the set of monomials 
where $\frac{\partial f}{\partial L}$ attains its minimum at $p$ is 
$\mathcal{F}_i(p)$. But, since $L$ separates $\mathcal{F}_i^+$ and 
$\mathcal{F}_i^-$, all the monomials in $\mathcal{F}_i(p)$ yield the same sign, 
so $p\notin \mathcal{T}_\R(\frac{\partial f}{\partial L})$ and $p$ is not 
singular.

Assume now that $p$ is not a singular point. Then, there exists a $L$ such that 
$p\notin \mathcal{T}_\R(\frac{\partial f}{\partial L})$. The monomials where 
the 
minimum of $\frac{\partial f}{\partial L}$ is attained at $p$ is one of the 
sets 
of the flag $\mathcal{F}_i(p)$. Since $p$ is not in the tropical variety 
defined 
by this derivative, it follows that $L$ separates $\mathcal{F}_i^+$ and 
$\mathcal{F}_i^-$.
\end{proof}

\begin{example}
Consider the real tropical cubic \[f = 0^+v^3\oplus 0^+w^3\oplus 0^-v^2w\oplus 
0^-vw^2\oplus 0^+v^2\oplus 0^+w^2\oplus 0^+vw\oplus 0^-v\oplus 0^-w\oplus 0^+\] 
The weight classes in the positive orthant are: $\{(0^+,0^+)\}$, $\{(a^+, a^+)| 
a<0\}$, $\{(0^+,a^+)| a>0\}$, $\{(a^+,0^+)| a>0\}$. The points in the weight 
class $\{(a^+, a^+)| a<0\}$ are not singular, since they do not belong to the 
real tropical variety defined by $\frac{\partial f}{\partial v+w-3}$. The 
points 
of the form $\{(0^+,a^+)| a>0\}$ are all singular, 
$\mathcal{F}_0^+=\{1,v^2,v^3\}$, $\mathcal{F}_0^-=\{v\}$. No $L$ separates 
$\mathcal{F}_0^+$ and $\mathcal{F}_0^-$ and $(0^+,a^+)$ belongs to the tropical 
curve defined by $\frac{\partial f}{\partial w} = 0^+w^3\oplus 0^-v^2w\oplus 
0^-vw^2\oplus  0^+w^2\oplus 0^+vw\oplus 0^-w$. Similarly $\{(a^+,0^+)| a>0\}$ 
are all singular points. Finally the point $(0^+,0^+)$ is also singular, as 
shown in Example~\ref{ex:real_radical_mala_trop}. For a point of the form 
$(0^+,a^+)$ the polynomial $F=2x^3 + (t^{2a} + t^a - 1)x^2y + (-2t^a - 2)xy^2 + 
y^3 + 2x^2 + 2xy + 2y^2 + (-10)x + (-t^a - 1)y + 6$, with $trop(F)=f$ has a 
singularity in $(1,t^a)$ and a similar result can be obtained for $(a^+,0^+)$.
\end{example}

\begin{example}
Let $p=(0^+,0^+)\in \TR^2$. We would like to describe the real tropical curves 
with a singularity at $p$. We only describe its maximal cones. Such a 
description in the complex case appears in \cite{sing-fixed-point} so we just 
discuss which are the valid distribution of signs on the coefficient of the 
tropical polynomial $f$. Up to symmetry and swap of signs the (maximal) 
combinatorial types of polynomials with a singularity at $p$ are.
\begin{enumerate}
\item $\mathcal{F}_0(p)$ is a two-dimensional circuit (four points, each three 
affinely independent). Then, the positive and negative monomials of the circuit 
cannot be separated by a line.
\item $\mathcal{F}_0(p)$ is a one-dimensional circuit (three collinear points) 
contained in a line $L$. The vertices of the circuit are of the same sign and 
the interior monomial is of the opposite sign. In this case, $\mathcal{F}_1$ 
consists on two monomials, we distinguish two cases:
\begin{enumerate}
\item The monomials of $\mathcal{F}_1$ lie on different halfspaces with respect 
to $L$. In this case, the two monomials have the same sign.
\item Both monomials of $\mathcal{F}_1$ lie on the same halfspace with respect 
to $L$. In this case, the two monomials have different signs.
\end{enumerate}
\end{enumerate}
\end{example}
\begin{center}
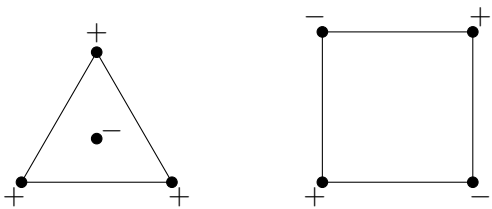
\begin{figure}
\begin{tikzpicture}
\draw (-1,0)--(1,0)--(0,1.73)--(-1,0);
\draw[fill] (0,0.58) circle(2pt) (-1,0) circle(2pt) (1,0) circle(2pt) (0,1.73) 
circle (2pt);
\draw (-1.1,-0.2) node {$+$} (1.1,-0.2) node {$+$} (0,1.99) node {$+$};
\draw (0.2,0.68) node {$-$};
\draw (3,0) -- (5,0) -- (5,2) -- (3,2) -- (3,0);
\draw[fill] (3,0) circle(2pt) (5,0) circle(2pt) (5,2) circle(2pt) (3,2) 
circle(2pt) (3,0);
\draw (2.9,-0.2) node {$+$} (5.1,2.2) node {$+$};
\draw (5.1,-0.2) node {$-$} (2.9,2.2) node {$-$};
\end{tikzpicture}
\caption{Possible distribution of signs on a circuit of dimension 
2.}\label{fig:circuit2}
\end{figure}
\end{center}

\begin{center}
\begin{figure}
\begin{tikzpicture}
\draw (0,1) -- (0,0) -- (0,-1);
\draw[fill] (0,1) circle(2pt) (0,0) circle(2pt) (0,-1) circle(2pt);
\draw (-1,0) circle(2pt) (1,0) circle(2pt);
\draw (0.2,0.2) node {$-$} (0.2,1.2) node {$+$} (0.2,-0.8) node {$+$};
\draw (1.2,0.2) node {$\pm$} (-0.8,0.2) node {$\pm$};

\draw (5,1)--(5,0)--(5,-1);
\draw[fill] (5,1) circle(2pt) (5,0) circle(2pt) (5,-1) circle(2pt);
\draw (6,-0.5) circle(2pt) (6,0.5) circle(2pt);
\draw (5.2,0.2) node {$-$} (5.2,1.2) node {$+$} (5.2,-0.8) node {$+$};
\draw (6.2,0.7) node {$\pm$} (6.2,-0.3) node {$\mp$};
\end{tikzpicture}
\caption{Possible distribution of signs around a circuit of dimension 1. The 
monomials in $\mathcal{F}_1$ are the white circles.}\label{circuit1}
\end{figure}
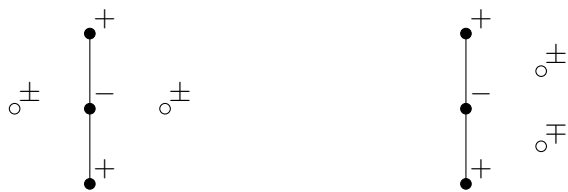
\end{center}

\end{document}